 \documentclass[11pt,a4paper,twoside]{paper}
\usepackage[T1]{fontenc}
\usepackage{hyperref}
\usepackage{amsmath,amssymb,amsthm}
\usepackage{parskip}
\newtheoremstyle{etoile}{\parskip}{\parskip}{\itshape}
                        {0pt}{\bfseries\sffamily}{.}{ }{}
\theoremstyle{etoile}

\newcommand\bbox{\hfill\rule{2mm}{2mm}}

\newtheorem{prop}{Proposition}[section]

\newtheorem{theo}[prop]{Theorem}
\newtheorem{defin}[prop]{Definition}
\newtheorem{lem}[prop]{Lemma}

\makeatletter
\@addtoreset{equation}{section}
\makeatother

\newcommand\egaldef{\stackrel{\mbox{\upshape\tiny def}}{=}}

\newcommand\1{\leavevmode\hbox{\rm \small1\kern-0.35em\normalsize1}}
        
\newcommand\EE{\mathsf{E}}
\newcommand\Dc{\mathcal{D}} 
\newcommand\Gb{\mathbf{G}}
\newcommand\Ac{\mathcal{A}}

\newcommand\Cc{\mathcal{C}}
\newcommand\CT{\widetilde{\Cc[T]}}
\newcommand\Lc{\mathcal{L}}
\newcommand\Mc{\mathcal{M}}
\newcommand\Oc{\mathcal{O}}
\newcommand\Rc{\mathcal{R}}
\newcommand\Sb{\mathbf{S}}
\newcommand\Vc{\mathcal{V}}
\newcommand\ZZ{\mathbb{Z}}
\newcommand\n{^{\scriptscriptstyle (N)}}
\newcommand\nd{^{\scriptscriptstyle (2N)}}
\newcommand\nn{\scriptscriptstyle N}
\def\DD{\displaystyle} 
\DeclareMathOperator*{\lra}{\leftrightarrows}
\DeclareMathOperator*{\rla}{\rightleftarrows}

\begin{document}
\title{About Hydrodynamic Limit of Some Exclusion Processes via Functional Integration}

\author{Guy Fayolle \thanks{INRIA - Domaine de Voluceau, 
Rocquencourt
BP 105 - 78153 Le Chesnay Cedex - France. Contact:
\texttt{Guy.Fayolle@inria.fr}} 
\and Cyril Furtlehner \footnotemark[1]  \and \hspace{10cm} May 2011}

\maketitle
\begin{abstract}
This article considers some classes of models dealing with the dynamics of discrete curves subjected to stochastic deformations. It turns out that the problems of interest can be set in terms of interacting exclusion processes, the ultimate goal being to derive  hydrodynamic limits  after proper scalings.  A seemingly new method is proposed, which relies on the analysis of specific partial differential operators, involving variational calculus and functional integration: indeed, the variables are the values of some functions at given points, the number of which tends to become infinite, which requires the construction of \emph{generalized measures}. 
 Starting from a detailed analysis of the \textsc{asep} system on the torus $\ZZ/N\ZZ$, we claim that the arguments a priori work in higher dimensions (\textsc{abc}, multi-type exclusion processes, etc),
leading to sytems of coupled partial differential equations of Burgers' type.
\end{abstract}

\paragraph{Keywords} Cauchy problem, exclusion process, functional integration, hydrodynamic limit, martingale, weak solution.


\section{Preliminaries}\label{INTRO}
 Interplay between discrete and continuous description is a recurrent   question
in statistical physics, which  in some cases can be answered quite rigorously via probabilistic methods. In the context of reaction-diffusion 
systems, this is tantamount to studying  fluid or hydrodynamics limits.   
Number of approaches have been proposed, in particular in the framework of exclusion 
processes, see e.g. \cite{Li},\cite{MaPr} \cite{Sp}, \cite{KiLa} 
and references therein.  As far as fluid or hydrodynamic limits are at stake, most of these methods have in common to be limited to systems for which the stationary states are  
given in closed product forms, or at least for which the invariant measure for finite $N$ (the size of the system) is explicitly known. For instance,  \textsc{asep} with open boundary can be described in terms of matrix product form (a sort of a non-commutative product form) and the continuous  limits can be understood by means of brownian bridges (see \cite{DeEvHaPa}). We propose to address this question from the following different point of view: starting from discrete sample paths subjected to stochastic deformations, the ultimate goal is to understand the nature of the limit curves when $N$ increases to infinity. How do these curves evolve with time, and which limiting process do they represent as $t$ goes to infinity (equilibrium curves)? Following \cite{FaFu,FaFu2,FaFu3},  we plan to give some  partial answers  to these problems in a series of forthcoming papers.

The method proposed in the present study is applied in detail to the \textsc{asep} model. The mathematical kernel relies on the analysis of specific partial differential equations involving variational calculus. A usual sequence of empirical measures is shown to converge in probability to a deterministic measure, which is the unique weak solution of a Cauchy problem. Here variables are in fact the values of some function at given points, and their number becomes infinite.  

In our opinion, the approach presents some new features, and very likely extend to  higher dimensions, namely multi-type exclusion processes. A future concern will be to establish a complete  hierarchy of systems of hydrodynamic equations, the study of which should allow us to describe non-Gibbs states. All these questions form also the matter of ongoing works.

\section{Model definition}
\subsection{A general stochastic clock model}\label{sec:clock}
Consider an oriented sample path of a planar random walk in $\Rc^2$, consisting of $N$ \emph{steps} (or \emph{links}) of equal size. Each step can have $n$ discrete possible orientations, drawn from the set  of angles with some given origin $\{\theta_k=\frac{2k\pi}{n}, k=0,\ldots, n-1\}$. The stochastic dynamics in force consists in displacing one single point at a time without breaking the path, while keeping all  links within the set of  admissible orientations. In this operation, two links are simultaneously displaced, what constrains quite strongly the possible dynamical rules 
\subsubsection{Constructing a related continuous-time Markov chain} 
\begin{itemize}
\item Jumps are produced by independent exponential events. 
\item  Periodic boundary conditions will be assumed, this point being not a crucial restriction. 
\item Dynamical rules are given by a set of \emph{reactions} between consecutive links, an equivalent formulation being possible in terms of \emph{random grammar}.
\end{itemize}
With each link is associated a \emph{type}, i.e. a {\emph{letter of an alphabet}. Hence, for any $n$, we  can define
which later on  will be sometimes referred to as a \emph{local exchange} process.

For $i \in [1,N]$  and $k \in [1,n]$, let $X_i^k$ represent a link of  type $k$ at site $i$. 
Then we can define the following set of reactions.
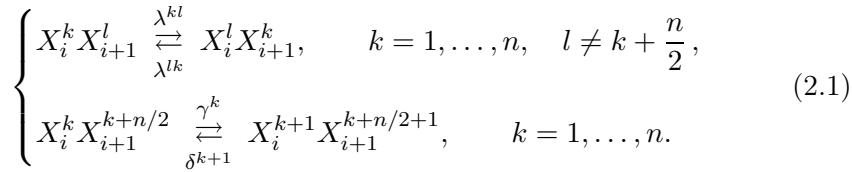
\begin{equation} \label{REAC}
\begin{cases}
 \DD X_i^k X_{i+1}^l\ \rla_{\lambda^{lk}}^{\lambda^{kl}}\ X_i^l X_{i+1}^k,\qquad
k=1,\ldots,n, \quad l \ne k+\frac{n}{2}\,,\\[0.4cm] 
\DD X_i^k X_{i+1}^{k+n/2}\ \rla_{\delta^{k+1}}^{\gamma^k} 
\ X_i^{k+1} X_{i+1}^{k+n/2+1}, \qquad k=1,\ldots,n.
\end{cases} 
\end{equation}
The red equations does exist only for even $n$, because of the existence of  \emph{folds} [two consecutive links with opposite directions], which yield a richer dynamics.

$X_i^k$ can also be viewed as a binary random variable describing the occupation of site 
$i$ by a letter of type $k$. Hence, the state space} of the system is represented by the array   \[
\eta \egaldef \{ X_i^k, i = 1,\ldots, N; k = 1,\ldots, n\}.
\]

\subsection{Examples} \label{sec:exemple}
\emph{(1) The simple exclusion process}

The first elementary and most sudied example is the simple exclusion process: this model, after mapping particles onto links, corresponds to a one-dimensional fluctuating interface. Here we take a binary alphabet and, letting
$X^1=\tau$ and $X^2=\bar\tau$, the set of reactions simply rewrites
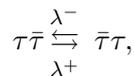
\[
\tau\bar\tau \lra_{\lambda^+}^{\lambda^-}\ \bar\tau\tau,
\]
where $\lambda^\pm$ are the transition rates for the jump of a particle to
the right or to the left. 

\emph{(2) The triangular lattice and the ABC model}

 Here the evolution of the random walk is  restricted to the triangular
 lattice. Each link (or step) of the walk is either $1$, $e^{2i\pi/3}$
 or $e^{4i\pi/3}$, and quite naturally will be said to be of type A, B or C, respectively. 
 This corresponds to the so-called \emph{ABC
 model}, since there is a coding by means of a  $3$-letter alphabet. 
 The set of \emph{transitions} (or reactions) is given by
\begin{eqnarray}
AB\ \lra_{p^+}^{p^-}\ BA, \qquad 
BC\ \lra_{q^+}^{q^-}\ CB, \qquad 
CA\ \lra_{r^+}^{r^-}\ AC, \qquad 
\end{eqnarray}
where there is a priori no symmetry, but we will impose \emph{periodic
boundary conditions} on the sample paths. This model was first
introduced in \cite{EvFoGoMu} in the context of particles with exclusion, and for some cases corresponding to the reversibility of the process, a Gibbs form for the invariant measure was given in \cite{EvKaKoMu}

\section{Hydrodynamics for the basic asymmetric exclusion process (ASEP)}
As mentioned above, we aim at obtaining hydrodynamic equations for a  class of exclusion  models. The method, although relying on classical powerful tools  (martingales, relative compactness of measures, functional analysis), has some new features which should hopefully prove fruitful in other contexts. The essence of the approach is in fact contained in the analysis of the popular \textsc{asep} model, presented below. We note the difficulty to find in the existing literature a complete study encompassing  various special cases (symmetry, total or weak asymmetry, etc). 

Consider $N$ sites labelled from $1$ to $N$, forming a discrete closed curve in the plane, so that  the numbering of sites is implicitly  taken modulo $N$, i.e. on the discrete torus $\Gb\n\egaldef \ZZ/N\ZZ$.  In higher dimension, say on the lattice $\ZZ^k$, the related  set of sites would be drawn on the torus  $(\ZZ/N\ZZ)^k$. 

We gather below some notational material valid throughout this paper.
\begin{itemize}
\item $\Rc$ (resp. $\Rc^+$) stands for the real (resp. positive real) line. $\Cc^k[0,1]$ is the collection of all real-valued, $k$-continuously differentiable functions defined on the interval $[0,1]$, and $\Mc$  is the space of all  finite positive measures on the torus $\Gb\egaldef [0,1)$. 
\item For $\Sb$ an arbitrary metric space, $\mathcal{P}(\Sb)$ is the set of probability measures on  $\Sb$, and $\Dc_\Sb[0,T]$ is the space of right continuous functions $z : [0,\infty]\to \Sb$ with left limits and $t\to z_{t}$. 
 \item $\Cc^\infty_{0}(K)$ is  the space of infinitely differentiable functions with compact support included in $K\in\Sb$, and we shall  write $\CT$ to denote the subset  of functions  $\phi(x,t)\in\Cc^\infty_{0}([0,1]\times[0,T_-])$ vanishing at $t=T$.
\item For $ i=1,\ldots,N$, let  $A_i\n(t)$ and $B_i\n(t)$ be binary random variables representing respectively a particule or a hole at site $i$, so that, owing to the exclusion constraint,  $A_i\n(t)+B_i\n(t)=1$, for all $1\le i\le N$. Thus $\bigl\{\mathbf{A}\n(t)\egaldef\bigl(A_i\n(t), \ldots,A_N\n(t)\bigr), t\ge0\bigr\}$ is a Markov  process.
\item  $\Omega\n$ will denote the generator of the Markov process $\mathbf{A}\n(t)$, and $\mathcal{F}_t\n=\sigma\bigl(\mathbf{A}\n(s),s\leq t\bigr)$ is the associated  natural filtration.
\item Our purpose is to  analyze the sequence of empirical random measures 
\begin{equation} \label{eq:emp1}
\mu\n_{t} = \frac{1}{N}\sum_{i\in\Gb\n} A_i\n(t) \delta _\frac{i}{\nn},
\end{equation}
when $N\to\infty$, after a convenient scaling of the parameters of the generator 
$\Omega\n$. The probability distribution associated with the path of the Markov process
$\mu\n_t, t\in[0,T]$, for some fixed $T$, will be simply denoted by $Q\n$.
 \end{itemize}
As usual, one can embed $\Gb\n$ in  $\Gb$, so that  a point  $i\in \Gb\n$ corresponds to the point  $i/N$ in $\Gb$. Hence, in view of (\ref{eq:emp1}), it is quite natural to let the sequence $Q\n$ be defined on a unique space  $\Dc_\Mc[0,T]$, which becomes a polish space (i.e. complete and separable) via the usual Skorokod topology, as soon as $\Mc$ is itself Polish (see e.g. \cite{EtKu}, chapter $3$). Without further comment,  $\Mc$ is assumed to be endowed with the vague product topology, as a consequence of the famous  Banach-Alaoglo and Tychonoff  theorems (see e.g. \cite{RUD,Ka}). 

Choose two arbitrary functions $\phi_a,\phi_b\in \CT$ and define the following real-valued positive measure
\begin{equation} \label{eq:emp2}
Z\n _t[\phi_a,\phi_b]  \egaldef \exp\biggl[\frac{1}{N}\sum_{i\in\Gb\n} \phi_a\Bigl(\frac{i}{N},t\Bigr)A_i\n(t)+\phi_b\Bigl(\frac{i}{N},t\Bigr)B_i\n(t) \biggr],
\end{equation} 
viewed as a  functional of $\phi_a,\phi_b$. Since $A_i\n(t)+B_i\n(t)=1$, for $1\leq i\leq N$, the transform $Z\n _t$ is essentially a functional of the sole function  
$\phi_a-\phi_b$, up to a constant uniformly bounded in $N$. Nevertheless, it will appear later that we need 2 independent functions. For the sake of brevity, the explicit dependence on 
$N,t$ or $\phi$, of quantities like $A_i\n(t), B_i\n(t), 
Z\n_t[\phi_a,\phi_b]$, will frequently be omitted, wherever the meaning remains clear from the context: for instance, we simply write $A_i, B_i$ or $Z\n_t$. Also $Z\n$ stands for the process $\{Z\n_t, \,t\ge 0\}$.

A standard powerful method to prove the convergence (in a sense to be specified later) of  the sequence of probability measures introduced in (\ref{eq:emp1}) consists  first in showing its relative compactness, and then in verifying  the coincidence of all possible limit points (see e.g. \cite{KiLa}. Moreover here, by the choice of the functions $\phi_a, \phi_b$, it suffices to prove these two properties for the sequence of projected measures defined on  $\Dc_\Rc[0,T]$ and corresponding to the processes $\{Z\n_t[\phi_a,\phi_b], t\ge 0\}$.

Let us now introduce  quantities which, as far as scaling is concerned, are crucial
in order to obtain meaningful hydrodynamic equations.
\begin{equation}\label{eq:scale1}
\begin{cases}
\DD\lambda(N) \egaldef \frac{\lambda^{ab}(N)+\lambda^{ba}(N)}{2} ,\\[0.2cm]
\mu(N)  \egaldef \lambda^{ab}(N)-\lambda^{ba}(N),
\end{cases}
\end{equation}
where the dependence of the rates on $N$ is explicitly mentioned.

\begin{theo} \label{theo:main} Let the system (\ref{eq:scale1})  have a given asymptotic expansion of the form, for large $N$,
\begin{equation}\label{eq:scale2}
\begin{cases}
\DD\lambda(N) \egaldef \lambda N^2 +o(N^2) ,\\[0.2cm]
\mu(N)  \egaldef \mu N + o(N),
\end{cases}
\end{equation}
 where  $\lambda$ and $\mu$ are fixed constants. [As for the scaling assumption (\ref{eq:scale2}),  the random measure $\log Z\n_t$ is a functional of the underlying Markov process, in which the time has been speeded up by a factor $N^2$ and the space shrunk by $N^{-1}$]. Assume also the sequence of initial empirical measures $\log Z\n _0$, taken at time $t=0$, converges in probability to some deterministic measure with a given density $\rho(x,0)$, so that, in probability,
  \begin{equation}\label{eq:init}
  \lim_{N\to\infty} \log Z\n _0 = \int_0^1[\rho(x,0)\phi_a(x,0)+(1-\rho(x,0))\phi_b(x,0)]dx ,
 \end{equation}
for any pair of  functions $\phi_a, \phi_b \in \CT$.

Then,  for every $t>0$, the sequence of random measures $\mu\n_t$ converges in probability, as  $N\to\infty$, to a deterministic measure having a density $\rho(x,t)$ with respect to the Lebesgue measure, which is the unique \emph{weak solution of  the   Cauchy problem}

\begin{equation}\label{eq:Cauchy}
\begin{split}
 \int_0^T \! \int_0^1 \left[\rho(x,t) \Bigl(\frac{\partial\phi(x,t)}{\partial t} +\lambda\frac{\partial^2\phi(x,t)}{\partial x^2}\Bigr) -\mu\rho(x,t)\bigl(1-\rho(x,t)\bigr)\frac{\partial\phi(x,t)}{\partial x}\right]\! dxdt 
 \\[0.2cm]
 + \int_0^1  \rho(x,0)\phi(x,0) dx  = 0, \hspace{1cm}
\end{split}
 \end{equation}
where  (\ref{eq:Cauchy}) holds for any function $\phi\in\CT$. 

 If, moreover, one assumes the existence of $\frac{\partial^2\rho(x,0)}{\partial x^2}$, for $\rho(x,0)$ given, then (\ref{eq:Cauchy}) reduces to a classical Burgers' equation
\[
\frac{\partial\rho(x,t)}{\partial t} = \lambda\frac{\partial^2\rho(x,t)}{\partial x^2} +
\mu [1-2\rho(x,t)] \frac{\partial\rho(x,t)}{\partial x}.
\]
\end{theo}

\begin{proof} \mbox{ } The proof  is contained in the next three subsections.

\subsection{Existence of limit points: sequential compactness} \label{TIGHT}
As usual in problems dealing with convergence of sequences of probability measures, our very starting point will be to establish the weak relative compactness of the set $\{\log Z\n_t, N\ge1\}$. Some of the probabilistic arguments employed in this paragraph are classical and can be found in the literature, e.g. \cite{Sp, KiLa}, although for slightly different or simpler models.

Letting  $\phi_a, \phi_b$ be two arbitrary functions in $\CT$, we refer to equation (\ref{eq:emp2}).

Using the exponential form of $Z\n_t$ and Lemma [A1-5.1] in \cite{KiLa} (see also  chapter 3 in \cite{EtKu} for related calculus), one can easily check that the two following random processes
\begin{eqnarray}
U_t\n & \egaldef &  Z\n_t - Z\n_0 - \int_0^t \bigl(\Omega\n [Z\n_s] + \theta_s\n Z\n_s\bigr) ds , 
\label{eq:martin1} \\
V_t\n &\egaldef & (U\n_t)^2  - \int_0^t \Bigl(\Omega\n  [(Z\n_s)^2]
- 2 Z\n_s \Omega\n  [Z\n_s] \Bigr)ds \label{eq:martin2}
\end{eqnarray}
are bounded $\{\mathcal{F}_t\n\}$-martingales, where 
\begin{equation}\label{eq:partial_t}
\theta_t\n \egaldef \frac{1}{N}\sum_{i\in\Gb\n} \Bigl[ \frac{\partial\phi_a}{\partial t}\Bigl(\frac{i}{N},t\Bigr)A_i\n(t)+ \frac{\partial\phi_b}{\partial t}\Bigl(\frac{i}{N},t\Bigr)B_i\n(t)\Bigr].
\end{equation}

Setting now
\begin{eqnarray*}\label{eq:delta}
\psi_{xy} & \egaldef & \phi_x-\phi_y=-\psi_{yx} , \\[0.2cm]
\Delta\psi_{xy} \Bigl(\frac{i}{N},t\Bigr) &\egaldef & \psi_{xy}\Bigl(\frac{i+1}{N},t\Bigr) -\psi_{xy}\Bigl(\frac{i}{N},t\Bigr),  \\[0.2cm]
 \widetilde{\lambda}\n_{xy}(i,t)& \egaldef & \lambda_{xy}(N)\left[
 \exp\biggl(\frac{1}{N} \Delta\psi_{xy} \Bigl(\frac{i}{N},t\Bigr)\biggr) - 1\right], 
\quad \ xy = ab \ \text{or} \ ba , 
\end{eqnarray*}
we have
\begin{equation}\label{eq:gen1}
 \Omega\n [Z\n_t ]  =    L\n_t Z\n_t ,
 \end{equation}
 where
\begin{equation}\label{eq:gen2}
L\n_t  \egaldef  \sum_{i\in\Gb\n} \widetilde{\lambda}\n_{ab}(i,t) A_{i}B_{i+1} +  
 \widetilde{\lambda}\n_{ba}(i,t) B_{i}A_{i+1}.
\end{equation}
  
On the other hand, a straightforward calculation in equation (\ref{eq:gen2}) allows to rewrite  (\ref{eq:martin2}) in  the form
\begin{equation}\label{eq:martin3}
V_t\n = (U\n_t)^2  - \int_0^t (Z\n_s)^2R_s\n ds ,
\end{equation}
where the process $R_t\n$ is stricly positive and given by
\[
R_t\n =  \sum_{i\in\Gb\n} \frac{[\widetilde{\lambda}\n_{ab}(i,t)]^2}{\lambda_{ab}(N)}
A_{i}B_{i+1} \,+\, \frac{[\widetilde{\lambda}\n_{ba}(i,t)]^2}{\lambda_{ba}(N)} B_{i}A_{i+1}.
\]  
 The integral term in  (\ref{eq:martin3}) is nothing else but the increasing process associated with Doob's decomposition of the submartingale  $(U\n_t)^2$. 

The folllowing estimates are crucial.
\begin{lem} \label{lem:esti}
\begin{eqnarray}
L\n_t  & = & \Oc(1) ,  \label{eq:lem1}  \\[0.2cm]
R_t\n & = & \Oc\bigl(\frac{1}{N}\bigr). \label{eq:lem2}
\end{eqnarray}
\end{lem}
\begin{proof} We will derive (\ref{eq:lem1}) by estimating  the right-hand side member of equation (\ref{eq:gen2}). From now on, for the sake of shortness, the first and second partial derivatives of $\psi (z,t)$ with respect to $z$ will be denoted respectively by $\psi '(z,t)$ and $\psi''(z,t)$.

Clearly,  $\Delta\psi_{xy} \bigl(\frac{i}{N},t\bigr) = \frac{1}{N}\psi_{xy} '\bigl(\frac{i}{N},t\bigr) + \mathcal{O}\bigl( \frac{1}{N^2}\bigr)$. Then, taking a  second order expansion of the exponential function and using equations (\ref{eq:scale1}) and (\ref{eq:scale2}), we can rewrite  (\ref{eq:gen2}) as
\begin{align}\label{eq:estim1}
 L\n_t & =  \frac{\mu(N)}{N} \sum_{i\in\Gb\n}\left[  \frac{A_{i}+A_{i+1}}{2} - A_{i}A_{i+1} \right] \Delta\psi_{ab} \Bigl(\frac{i}{N},t\Bigr) \nonumber \\
&+  \frac{\lambda(N)}{N}\sum_{i\in\Gb\n} (A_{i}-A_{i+1})\Delta\psi_{ab} \Bigl(\frac{i}{N},t\Bigr)
 +  \Oc \Bigl(\frac{1}{N}\Bigr).
\end{align}

 The first  sum on the right in (\ref{eq:estim1}) is uniformly bounded by a  constant depending on 
 $\psi$. Indeed $|A_{i}| \le 1$, and $\psi'$ is of bounded variation since $\psi\in\CT]$.
  As for the second sum coming in  (\ref{eq:estim1}), we have
  \[ 
  \! \sum_{i\in\Gb\n} \! (A_{i}-A_{i+1})\Delta\psi_{ab} \Bigl(\frac{i}{N},t\Bigr) = \!\!
   \sum_{i\in\Gb\n} \! A_{i+1} \left[\Delta\psi_{ab} \Bigl(\frac{i+1}{N},t\Bigr) -
   \Delta\psi_{ab} \Bigl(\frac{i}{N},t\Bigr)\right] \! .
   \]
Then the discrete Laplacian 
\[\Delta\psi_{ab} \Bigl(\frac{i+1}{N},t\Bigr) - \Delta\psi_{ab} \Bigl(\frac{i}{N},t\Bigr) \equiv 
\psi_{ab}\Bigl(\frac{i+2}{N},t\Bigr) -2\psi_{ab} \Bigl(\frac{i+1}{N},t\Bigr) + \psi_{ab} \Bigl(\frac{i}{N},t\Bigr)
\]
admits of the simple expansion
\begin{equation}\label{eq:psi}
\Delta\psi_{ab} \Bigl(\frac{i+1}{N},t\Bigr) - \Delta\psi_{ab} \Bigl(\frac{i}{N},t\Bigr) =
\frac{1}{N^2} \psi_{ab}''\Bigl(\frac{i}{N},t\Bigr) +\Oc\Bigl(\frac{1}{N^2}\Bigr).
\end{equation}

By (\ref{eq:scale2}),  $\DD\lambda(N)=\lambda N^2 +o(N^2)$,  
so that (\ref{eq:psi}) implies
 \begin{eqnarray}
 \frac{\lambda(N)}{N}\sum_{i\in\Gb\n} (A_{i}-A_{i+1})\Delta\psi_{ab} \Bigl(\frac{i}{N},t\Bigr)
&=& \sum_{i\in\Gb\n} \frac{\lambda A_{i+1}}{N} \psi_{ab}''\Bigl(\frac{i}{N},t\Bigr) + 
o\Bigl(\frac{1}{N}\Bigr) \nonumber \\[0.2cm]
& = &\Oc(1), \label{eq:estim2}
\end{eqnarray}
which concludes the proof of (\ref{eq:lem1}). The computation of $R_t\n$ leading to (\ref{eq:lem2}) can be obtained via similar arguments, remarking that
\[
R_t\n[\phi_a,\phi_b] = L_t\n[2\phi_a,2\phi_b] - 2L_t\n[\phi_a,\phi_b].  
\]
\end{proof}
To show the relative compactness of the family $Z\n$, which from the separability and the completeness of the underlying spaces is here equivalent to tightness, we proceed as in \cite{KiLa} by means of the following useful criterion.

\begin{prop} [Aldous's tightness criterion, see \cite{Bil}] \label{prop:tight}
A sequence $\{X\n\}$ of random elements of  $\Dc_\Rc[0,T]$ is tight (i.e. the distributions of the  $\{X\n\}$ are tight) if the two following conditions hold:

\begin{itemize} 
\item[(i) ]
\begin{equation} \label{AT1}
\lim_{a\to\infty} \limsup_N P [ || X\n ||  \ge a] =0, 
\end{equation}
where $\DD || X\n || \egaldef \sup_{t\le T} |X\n_t|$.
\item[(ii)] For each $\epsilon, \eta$, there positive numbers $\delta_0$ and $N_0$, such that, if 
$\delta\le\delta_0$ and $N\ge N_0$, and if $\tau$ is an arbitrary stopping time with
$\tau+\delta\le T$, then
\begin{equation} \label{AT2}
    P\bigl[ | X\n_{\tau+\delta} - X\n_{\tau} | \ge \epsilon \bigr] \le \eta .
\end{equation}
\end{itemize}
Note that condition (\ref{AT1}) is always necessary for tightness. \bbox
\end{prop}

We shall now apply Lemma \ref{lem:esti} to equations (\ref{eq:martin1}) and  (\ref{eq:martin3}), the role of $X_t\n$ in Proposition \ref{prop:tight} being played by 
$Z_t\n$. 

The random variables $Z\n_t$ and $\theta_t\n$ are clearly uniformly bounded, so that condition (\ref{AT1}) is immediate. To check condition (\ref{AT2}), rewrite (\ref{eq:martin1}) as
\begin{equation} \label{eq:martin4}
Z\n_{t+\delta} - Z\n_t = U\n_{t+\delta} - U\n_t + \int_{t}^{t+\delta}\  (L\n_s +\theta_s\n) 
Z\n_s ds .
\end{equation}
For $t\in[0,T]$, the integral term in (\ref{eq:martin4}) is bounded in modulus by $K\delta$ 
($K$~being a constant uniformly bounded in $N$ and $\psi$), and it satisfies (\ref{AT2}),
whenever  $t$ is replaced by an arbitrary stopping time. We are left with the analysis of $U\n_t$. But, from  (\ref{eq:martin3}), (\ref{eq:lem2}) and Doob's inequality for sub-martingales, we have
\begin{eqnarray}
\EE\bigl[(U\n_{t+\delta} - U\n_t )^2\bigr] &=& \EE\left[\int_t^{t+\delta} (Z\n_s)^2R_s\n ds\right]  \le  \frac{C\delta}{N}  ,  \nonumber \\[0.2cm]
P\left[\sup_{t\le T}  | U\n_t | \ge \epsilon\right] &\le& \frac{4}{\epsilon^2} 
\EE \left[\int_0^T (Z\n_s)^2R_s\n ds\right] \le \frac{4CT}{N\epsilon^2}, \label{eq:Doob}
\end{eqnarray}
where $C$ is a positive constant depending only on $\psi$. Thus $U\n_t\to 0$ in probability, as $N\to\infty$. This last property, together with (\ref{eq:martin1}), (\ref{eq:martin4}) and assumption (\ref{eq:init}), yield (\ref{AT2}) and the announced (weak) relative compactness of the sequence $Z\n_t$. Hence, the sequence of probability measures $Q\n$, defined on $\Dc_\mathcal{M}[0,T]$ and corresponding to the process $\mu\n_t$, is also relatively compact: this is a consequence of  classical projection theorems (see for instance Theorem 16.27 in \cite{Ka}). We are now in a position to establish a  further important property.

Let  $Q$ the  limit point of some  arbitrary subsequence $Q^{(N_k)}$, as $N_k\to\infty$, and  $Z_t\egaldef \lim_{{N_k}\to\infty}Z^{(N_k)}_t$. Then  the support of $Q$ is a set of sample paths absolutely continuous with respect to the Lebesgue measure. Indeed, the application $\mu_t\to \sup_{t\le T}\log Z_t$ is continuous and we have the immediate bound
\[
\sup_{t\le T}\log Z_t \le \int_0^1 [ |\phi_a(x,t) | +  |\phi_b(x,t)|]dx,
\]
which holds for all $\psi_{a}, \psi_{b} \in \mathbf{C}^2[0,1]$. Hence, by weak convergence, any limit point $Z_t$ has the form
\begin{equation}\label{eq:lim}
Z_t[\phi_a,\phi_b] = \exp\Bigl[\int_0^1 [\rho(x,t)\phi_a(x,t)+(1-\rho(x,t)\phi_b(x,t) ] dx\Bigr],
\end{equation}
where  $\rho(x,t)$  denotes the limit density [which a priori is a random quantity] of the sequence of empirical measures  $\mu_t^{(N_k)}$ introduced in  (\ref{eq:emp1}).

\subsection{A functional integral operator to characterize limit points} 
This is somehow the Gordian knot of the problem. Relying on the above weak compactness property, our next result shows that any arbitrary limit point $Q$ is concentrated on a set of trajectories which are weak solutions of an \emph{integral equation}. 

The main idea is to consider for a while the $2N$ quantities $\phi_a\bigl(\frac{i}{N},t\bigr), \phi_b\bigl(\frac{i}{N},t\bigr)$, $1\leq i \leq N$, as \emph{ordinary free variables}, which for the sake of shortness will be denoted respectively by $x\n_i$ and $y\n_i$. With this approach,  the problem of the hydrodynamic limit will appear to be mostly of an analytical nature.

Let
\begin{eqnarray*}
\alpha_{xy}\n(i,t) &= &\lambda_{ab}(N)\left[
 \exp\Bigl(\frac{ x\n_{i+1}  - x\n_i  + y\n_i  - y\n_{i+1}}{N} \Bigr) - 1\right], \\
 \alpha_{yx}\n(i,t) & = & \lambda_{ba}(N)\left[
 \exp\Bigl(\frac{ y\n_{i+1}  -y\n_i  + x\n_i  - x\n_{i+1}}{N} \Bigr) - 1\right].
 \end{eqnarray*}
Then, using  (\ref{eq:martin1}), (\ref{eq:partial_t}), (\ref{eq:gen1}), (\ref{eq:gen2}) and the definition of $Z\n_t$, we obtain immediately  the following   \emph{functional partial differential equation} (FPDE)
\begin{equation} \label{eq:deriv2}
  \frac{d(Z\n_t-U\n_t)}{dt}   = \Lc\n_t [Z\n_t] + \theta_t\n Z\n_t, 
\end{equation}
 where $\Lc\n_t$ is the operator
  \[
 \Lc\n_t [h] \egaldef  N^2\sum_{i\in\Gb\n}
\alpha_{xy}\n(i,t) \frac{\partial^2 h}{\partial x\n_i \partial y\n_{i+1}}
+ \alpha_{yx}\n(i,t)\frac{\partial^2 h}{\partial y\n_i \partial x\n_{i+1}}.
\]
More precisely,  introducing  the family of cylinder sets
\begin{equation}\label{eq:cylinder}
\Vc^{(p)} \egaldef [-|\Phi|,|\Phi|]^p, \quad p=1,2\ldots,
\end{equation}
with
\begin{equation}\label{eq:norm}
|\Phi| \egaldef \sup_{(z,t)\in[0,1]\times[0,T]} \bigl(|\phi_a(z,t)|, |\phi_b(z,t)|\bigr),
\end{equation}
we see immediately  that $\Lc\n_t$ acts on a subspace of $\Cc^\infty_0(\Vc\nd)$, since $Z\n_t$ is \emph{analytic} with respect to the coordinates 
$\{\phi_a(.,t),\phi_b(.,t)\}$, for each finite $N$ (things will be made more precise in section  \ref{OPERA}). Needless to say that  $\Lc\n_t$ is not of parabolic type, as the quadratic form associated with the second order derivative terms is clearly non definite (see e.g. \cite{EgSh}). In addition, equation (\ref{eq:deriv2})  is a well defined stochastic FPDE, as all underlying probability spaces emanate from families of interacting Poisson processes.  

Now, it might be worth accounting for the use of the word \emph{functional} above, and for the reason of isolating the third term on the right in (\ref{eq:deriv2}). Indeed, by (\ref{eq:partial_t}), $\theta_t\n$ is a functional involving also a \emph{partial} derivative 
of $Z\n_t$ with respect to $t$, as we can write
 \begin{equation} \label{eq:derivpart}
  \theta_t\n Z\n_t = \frac{1}{N}\sum_{i\in\Gb\n} \Bigl[  \frac{\partial Z\n_t}{\partial x\n_i} \frac{\partial  x\n_i}{\partial t} + \frac{\partial Z\n_t}{\partial y\n_i }\frac{\partial  y\n_i}{\partial t}\Bigr] =  \frac{\partial Z\n_t}{\partial t}.
 \end{equation}
The last equality in (\ref{eq:derivpart}) might look somewhat formal, but will come out more clearly in Section \ref{OPERA}. Note also, by the same argument which led to (\ref{eq:lim}), that we have 
\begin{equation}\label{eq:lim_t}
\lim_{N_k\to\infty} \theta_t^{(N_k)} = \int_0^1 \bigl[\rho(x,t)\frac{\partial\phi_a}{\partial t}(x,t)+(1-\rho(x,t) \frac{\partial\phi_b}{\partial t}(x,t) \bigr] dx .
\end{equation}

Our essential agendum is to prove  that any limit point of  the sequence of random measures $\mu_t \stackrel{weak}{=}\DD\lim_{N_k\to\infty}\mu_t^{(N_k)}$ satisfies an integral equation  corresponding to a \emph{weak solution} (or \emph{distributional} in Schwartz sense) of a Cauchy type operator.
To overcome the chief difficulty, namely the behaviour of the limit sum in (\ref{eq:deriv2}), we propose a seemingly new approach, which is tantamount to analyzing the family of second order linear partial differential operators $\Lc\n_t$ along the sequence $N_k\to\infty$. 

As briefly emphasized in the remark at the end of this section,  a brute force analysis of 
(\ref{eq:deriv2})  would lead to a dead end. Indeed an important preliminary step consists in extracting the juice of  the estimates obtained in Lemma \ref{lem:esti}, to rewrite the operator $\Lc\n_t$ in terms of only $N$ principal variables, up to quantities of order
$\Oc\Bigl(\frac{1}{N}\Bigr)$. This is summarized in the next lemma.

\begin{lem}\label{lem:esti2}
The following FPDE  holds.
\begin{equation} \label{eq:deriv4}
 \frac{d(Z\n_t-U\n_t)}{dt}  \egaldef \Ac_t\n[Z\n_t] + \theta_t\n Z\n_t + \Oc\Bigl(\frac{1}{N}\Bigr),
\end{equation}
where $\Ac_t\n$ is viewed as an operator with domain  $\Cc^\infty_{0}(\Vc^{\n})$ such that
\begin{equation} \label{eq:deriv5}
\begin{split}
\Ac_t\n[g] & \egaldef \sum_{i\in\Gb\n} \mu \psi_{ab}' \Bigl(\frac{i}{N},t\Bigr) \left[\frac{1}{2}\biggl( \frac{\partial g}{\partial x\n_i} + \frac{\partial g}{\partial x\n_{i+1}}\biggr) -N \frac{\partial^2 g}{\partial x\n_i \partial x\n_{i+1}} \right] \\[0.2cm]
& +\lambda \sum_{i\in\Gb\n} \psi_{ab}'' \Bigl(\frac{i}{N},t\Bigr) 
\frac{\partial g}{\partial x\n_{i+1}},
\end{split}
\end{equation}
the term $\Oc\bigl(\frac{1}{N}\bigr)$ being in modulus uniformly bounded by 
 $\frac{C}{N}$, where $C$ denotes a constant depending only on the quantity
 \[
 \sup_{(x,t)\in[0,1]\times [0,T]} \{\psi(x,t), \psi'(x,t), \psi''(x,t)\}\, .
 \]
 \end{lem} 
\begin{proof}
The result follows by elementary algebraic manipulations  from equations (\ref{eq:scale2}), (\ref{eq:martin1}), (\ref{eq:gen1}), (\ref{eq:estim1}), (\ref{eq:estim2}), and details will be  omitted.
\end{proof}
Taking Lemma \ref{lem:esti2} as a starting point, we sketch out below in Sections \ref{C1} and \ref{C2} the main lines of our analytical approach, which indeed can be briefly summarized by means of some  keywords.
\begin{itemize}
\item \emph{Coupling},  taken here in Skohorod's context.
\item \emph{Regularization} and  \emph{functional integration}.  Regularization refers to the fundamental method used in the theory of distributions to approximate either functions or distributions. We shall apply it to the FPDE (\ref{eq:deriv4}), considering  $t$ and the $N$ values \emph{of the function $\phi_a(.,t)$ (taken at points of the torus)} as $N+1$ \emph{ordinary variables}. Then, passing to the limit as $N\to\infty$, we  introduce convenient functional integrals together with \emph{variational derivatives}. This might likely extend to much wider systems, although this assertion could certainly be debated.
\end{itemize}
\subsubsection{Interim reduction to an almost sure convergence setting}\label{C1} This can be achieved by means of the extended Skohorod coupling (or transfer) theorem (see Corollary 6.12  in \cite{Ka}), which in brief says that, if a sequence of real random variables $(\xi_k)$ is such that $\lim_{k\to\infty}f_k(\xi_k)= f(\xi)$  converges in distribution, then there exist a probability space  $\mathcal{V}$ and  a new random sequence $\widetilde{\xi_k}$, such that  $\widetilde{\xi_k}\stackrel{\Lc}{=}\xi_k$ and $\lim_{k\to\infty} f_k(\widetilde{\xi_k})= f(\xi)$, almost surely in $\mathcal{V}$, with $\widetilde{\xi}\stackrel{\Lc}{=}\xi$. Here this theorem will be applied to the family   $Z^{(N_k)}_t$,  which thus gives rise to a new sequence, named $Y^{(N_k)}_t$  in the sequel. Clearly this step is in no way obligatory, but rather a matter of taste. Indeed, one could still keep on working in a weak convergence context, with Alexandrov's portmanteau theorem (see e.g. \cite{EtKu}) whenever needed. $\bbox$

\subsubsection{Considering (\ref{eq:deriv5}) as a partial differential operator with constant coefficients} \label{C2} 
For each finite $N$, we consider the quantities 
\[
 \psi_{ab}' \Bigl(\frac{i}{N},t\Bigr), \psi_{ab}'' \Bigl(\frac{i}{N},t\Bigr), \ i = 1,\dots,N,
\]
as \emph{constant parameters, while the  $x\n_i$'s will be taken as free variables from a variational calculus point of view}. This  is clearly feasible, remembering that, by definition, 
$\psi_{ab} = \phi_a-\phi_b$, for all $\phi_a, \phi_b \in \CT$.  Hereafter, $t$ will be viewed as an exogeneous mute variable, not participating concretely in the proposed variational approach.

Then, according to  the notation introduced in section \ref{C1},  we can rewrite  (\ref{eq:deriv4}) in the form 
\begin{equation}\label{eq:esti3}
 \frac{d(Y\n_t - U\n_t)}{dt} \egaldef \Ac_t\n[Y\n_t] + \theta_t\n Y\n_t + \Oc\Bigl(\frac{1}{N}\Bigr),
\end{equation}
where $\theta_t\n$ is still given by (\ref{eq:partial_t}), keeping in mind that all random variables in (\ref{eq:esti3}) are defined with respect to this new (though unspecified) probability space introduced in section \ref{C1} above. In particular, from the tightness proved in section \ref{TIGHT}, letting  $N\to\infty$ along some subsequence $N_k$, we have
\[ 
\lim_{k\to\infty} Y^{(N_k)} _t[\phi_a,\phi_b] \stackrel{a.s.}{\to}Y _t[\phi_a,\phi_b].
\]

\subsubsection{Analysis of the FDPE \protect (\ref{eq:esti3})} \label{OPERA} 
The idea now is to propose a \emph{regularization} procedure, which consists in carrying out the convolution of  (\ref{eq:esti3}) with a suitably chosen test function, noting that $Y\n _t$ nowhere vanishes and is uniformly bounded. 

As usual,  the convolution $f \star g$ of two integrable functions $f,g \in \Cc^\infty_0(\Vc\n)$, with $\Vc\n$ given in  (\ref{eq:cylinder}), will be defined by
 \[
 (f \star g)(u) = \int_{\Vc\n} f(u-v)g(v)dv .
 \]
Let $\omega$ be the function of the real variable $z$ defined by
\[
\omega(z)\egaldef
\begin{cases}
\exp\bigl( \frac{1}{z}\bigr)  \ \mathrm{if} \  z< 0, \\[0.2cm]
0 \ \mathrm{if} \  z \geq 0,
\end{cases}
\]
where it will be convenient to write $\omega'(z) \egaldef \frac{d\omega(z)}{dz}$ and 
$\omega''(z) \egaldef \frac{d^2\omega(z)}{dz^2}$.

Setting $\vec{x}\n \egaldef (x\n_1,x\n_2,\ldots,x\n_N)$, with $x\n_i=\phi_a\bigl(\frac{i}{N},t\bigr), 1\leq i\leq N$, we introduce the following  family of positive test functions 
 $\chi_{\varepsilon}\n\in\Cc^\infty_0(\Vc\n)$, 
\begin{equation}\label{eq:test}
\chi_{\varepsilon}\n(\vec{x}\n) = \omega\biggl(\frac{1}{N}\sum_{i=1}^N (x\n_i)^2 -\varepsilon^2\biggr), \quad \varepsilon\geq 0.
\end{equation}
\paragraph{Note:} \emph{It is worth keeping in mind that, as often as possible, the time variable $t$ will be omitted in most of the mathematical quantities,  e.g. $\vec{x}\n(t)$. Indeed, as mentioned before, $t$ plays in some sense the role of a \emph{parameter.}}

From (\ref{eq:esti3}), it follows immediately that
\begin{equation}\label{eq:conv1}
\begin{split}
\left(\frac{d(Y\n_t - U\n_t)}{dt} \star \chi_{\varepsilon}\n\right)(\vec{x}\n) =  & \left(\Ac_t\n[Y\n_t] \star \chi_{\varepsilon}\n\right)(\vec{x}\n) \\
 + & \, \bigl((\theta_t\n Y\n_t) \star \chi_{\varepsilon}\n\bigr)(\vec{x}\n) + \Oc \Bigl(\frac{1}{N}\Bigr).
\end{split}
\end{equation}
Now, by (\ref{eq:deriv5}) and according to the statement made in Section \ref{C2}, the first term in the right-hand side member of (\ref{eq:conv1}) can be integrated by parts, so that, for any finite $N$ and $\varepsilon$ sufficiently small,
\begin{equation}\label{eq:adj1}
 \left(\Ac_t\n[Y\n_t] \star \chi_{\varepsilon}\n\right)(\vec{x}\n) = 
 \left(\widetilde{\Ac}_t\n[\chi_{\varepsilon}\n] \star Y\n_t\right)(\vec{x}\n) ,
  \end{equation}
where $\widetilde{\Ac}\n_t$  is by definition the \emph{adjoint} operator of  $\Ac\n$ in the Lagrange sense. Here the domain of  $\widetilde{\Ac}\n_t$ consists of all functions  $h\n$ of the form
\[
h\n = \exp \left[\int_0^1 d\sigma_a\n(x)V[\phi_a(x,t)] +  d\sigma_b\n(x)V[\phi_b(x,t)]\right], 
\]
where 
\begin{itemize}
\item $V:  \CT\rightarrow\Rc^+$ stands for an arbitrary analytic function; 
\item  $\sigma_a\n$ and $\sigma_b\n$ are arbitrary \emph{discrete} probability measures on $\Gb\n$.
\end{itemize}
Clearly $\chi_{\varepsilon}\n$ belongs to the domain of $\widetilde{\Ac}\n_t$. Under the assumptions made in Section \ref{C2},  a direct integration by parts  in  (\ref{eq:deriv5})  yields the formula
\begin{equation} \label{eq:adj2}
\begin{split}
\widetilde{\Ac}\n_t[h]  & = - \sum_{i\in\Gb\n} \mu \psi_{ab}' 
\Bigl(\frac{i}{N},t\Bigr) \left[ \frac{1}{2}\biggl( \frac{\partial h}{\partial x\n_i} + 
\frac{\partial h}{\partial x\n_{i+1}}\biggr) + N \frac{\partial^2 h}{\partial x\n_i \partial x\n_{i+1}}\right]  \\[0.2cm] 
& - \lambda \sum_{i\in\Gb\n} \psi_{ab}'' \Bigl(\frac{i}{N},t\Bigr) 
\frac{\partial h}{\partial x\n_{i+1}},
\end{split}
\end{equation}
where $\widetilde{\Ac}\n_t$ has been defined in (\ref{eq:adj1}).

Let, for each $\phi\equiv\phi_a\in\CT$,
\begin{equation}\label{eq:chi-lim}
\chi_{\varepsilon}(\phi) \egaldef \lim_{N\to\infty} \chi_{\varepsilon}\n(\vec{x}\n) = 
\omega\left(\int_0^1 \phi^2(x,t)dx -\varepsilon^2 \right), 
\end{equation}
where the integral in (\ref{eq:chi-lim}) is readily obtained as the limit of 
the Riemann sum in (\ref{eq:test}).

 \begin{lem}\label{OP1}
For each $\phi(x,t)\in\CT$, the following limit holds uniformly.
\begin{equation}\label{eq:A-lim}
\lim_{N\to\infty} \! \widetilde{\Ac}\n_t[\chi_{\varepsilon}\n] (\vec{x}\n) = - \!\!\int_0^1 \!\bigl[ \mu\psi_{ab}' (x,t) K(\phi,x,t) + \lambda\psi_{ab}'' (x,t)H(\phi,x,t)\bigr] dx,
\end{equation}
\end{lem}
with
\begin{equation}\label{DELTA}
\begin{cases}
\DD H(\phi,z,t) =2\phi (z,t) \omega' \bigg(\int_0^1 \phi^2(u,t)du -\varepsilon^2\biggr),
\\[0.3cm]
\DD K(\phi,z,t)= H(\phi,z,t) +  4\phi^2 (z,t)\omega'' \bigg(\int_0^1 \phi^2(u,t)du -\varepsilon^2\biggr).
\end{cases}
\end{equation}
 \begin{proof}
For fixed $N$, the function $\widetilde{\Ac}\n_t[\chi_{\varepsilon}\n]$ has partial derivatives of any order with respect to the coordinates  $x\n_i, i=1,\ldots,N$. In particular, we get 
from \eqref{eq:test}, for each $i\in\Gb\n$,
\begin{eqnarray*}
\frac{\partial\chi_{\varepsilon}\n}{\partial x\n_i} (\vec{x}\n)& = & 
2\frac{x\n_i}{N}\omega'\biggl(\frac{1}{N}\sum_{i=1}^N \bigl(x\n_i\bigr)^2 -\varepsilon^2 \biggr), \\[0.2cm]
N\frac{\partial^2\chi_{\varepsilon}\n}{\partial x\n_i\partial x\n_{i+1}}(\vec{x}\n) & = &
4\frac{x\n_ix\n_{i+1}}{N}\omega''\biggl(\frac{1}{N}\sum_{i=1}^N \bigl(x\n_i\bigr)^2 -\varepsilon^2 \biggr) .
\end{eqnarray*}
So we are again left with Riemann sums, which, for any continuous function $k(x,t)$, yield at once
\[
\begin{split}
& \lim_{N\to\infty} \sum_{i\in\Gb\n} k \Bigl(\frac{i}{N},t\Bigr) \frac{\partial\chi_{\varepsilon}\n}{\partial x\n_i} (\vec{x}\n) = \\
&2\omega' \bigg(\int_0^1 \phi^2(x,t)dx -\varepsilon^2\biggr) \int_0^1 k(x,t)\phi (x,t)dx ,
 \end{split}
\]
and
\[
\begin{split}
& \lim_{N\to\infty} N \sum_{i\in\Gb\n} k 
 \Bigl(\frac{i}{N},t\Bigr) \frac{\partial^2\chi_{\varepsilon}\n}{\partial x\n_i\partial x\n_{i+1}} (\vec{x}\n)  =   \\ 
& 4\omega'' \bigg(\int_0^1 \phi^2(x,t)dx -\varepsilon^2\biggr) \int_0^1 k(x,t)\phi^2 (x,t)dx . 
 \end{split}
\]
Keeping in mind that, for all $N$, the vector $\vec{x}\n$  (from its very definition) must be a discretization of some function element in $\CT$, the last important step to derive (\ref{eq:A-lim}) requires to give a precise meaning to the limit
\[ 
\lim_{N\to\infty}\int_{\Vc\n}\chi_{\varepsilon}\n(\vec{x}\n) d\vec{x}\n, 
\]
allowing to carry out  \emph{functional integration} and \emph{variational differentiation}. In this respect, let us emphasize (if necessary at all\,!) that the usual constructions of measures and integrals do not apply in general when the domain of integration is an infinite-dimensional space of functions or mappings, all the more because a complete axiomatic for functional integration does not really exist; indeed each case requires the construction of ad hoc \emph{generalized measures}, see promeasures in \cite{BOUR,CAR} or quasi-measures in \cite{MAI}. These questions prove to be of importance in various problems related to theoretical physics. In our case study, integration over paths belonging to $\in \CT$ needs to be properly constructed.  

The point is to define, as $N\to\infty$, a \emph{volume element}  denoted by
$\delta(\phi)$. We shall not do it here, but this could be achieved by mimicking classical fundamental approaches, see e.g.  \cite{BOUR,CAR,MAI}. The main tool is the important F.~Riesz' representation theorem, which  to every positive linear functional  let correspond a unique positive measure.
\end{proof}
Now from equation \eqref{eq:chi-lim} it is permissible to introduce the normalized test functional
 \[
 \overline{\chi}_{\varepsilon}(\phi) = \frac{\chi_{\varepsilon}(\phi)}{D},
\]
where $D$ is chosen to ensure
\[ 
\int_{\Cc^\infty_{0}([-|\Phi|,|\Phi|])}  \overline{\chi}_{\varepsilon}(\phi)\, \delta(\phi) = 1,
\]
and $\Phi$ is given by \eqref{eq:norm}, so that
\[
D= \int_{\Cc^\infty_{0}([-|\Phi|,|\Phi|])}  \omega\left(\int_0^1 \phi^2(x,t)dx -\varepsilon^2 \right) 
\,\delta(\phi).
\]
From  Skohorod's coupling theorem,  $Y_t$ does satisfy an equation of the form (\ref{eq:lim}). Hence, we can write  the following functional derivatives (which are plainly of a Radon-Nykodym nature)
\begin{equation} \label{eq:dY}
\begin{cases} 
\DD \frac{\partial Y_t}{\partial\phi} = \rho(.,t) Y_t  ,\\[0.3cm]
\DD \frac{\partial^2 Y_t}{\partial\phi^2} = \rho^2(.,t)Y_t .
\end{cases}
 \end{equation}

Now everything is in order to complete the puzzle, according to the following steps.
\begin{enumerate}
\item First, using  \eqref{eq:adj1}, rewrite \eqref{eq:conv1} as
\begin{equation} \label{eq:conv2}
\begin{split}
\left(\frac{d(Y\n_t - U\n_t)}{dt} \star \chi_{\varepsilon}\n\right)(\vec{x}\n) =  & 
 \left(\widetilde{\Ac}_t\n[\chi_{\varepsilon}\n] \star Y\n_t\right)(\vec{x}\n) \\
 + & \, \bigl((\theta_t\n Y\n_t) \star \chi_{\varepsilon}\n\bigr)(\vec{x}\n) 
 + \Oc \Bigl(\frac{1}{N}\Bigr).
\end{split}
\end{equation}
\item Let $N\to\infty$ in  \eqref{eq:conv2} and then replace
$\chi_{\varepsilon}(\phi)$ by $\overline{\chi}_{\varepsilon}(\phi)$, remembering that by \eqref{eq:Doob} $U\n_t =\Oc(1/N)\to0$ uniformly. 
\item Carry out two functional  integration by parts in equation \eqref{eq:A-lim} by making use of  
 \eqref{eq:dY}.
\item Finally, integrate on $[0,T]$, let $\varepsilon\to0$ and switch back to the original probability space, where $Z\n _t[\phi_a,\phi_b]$, by section \ref{C1}, converges in distribution to $Z_t[\phi_a,\phi_b]$: this yields exactly the announced Cauchy problem \eqref{eq:Cauchy}.
\end{enumerate}
Hence, the family of random measures $\mu\n_{t}$ converges in distribution to a deterministic measure $\mu_t$, which in this peculiar case implies also convergence in probability.
\subsection{Uniqueness} 
The problem of uniqueness  of weak solutions of the Cauchy problem (\ref{eq:Cauchy}) for nonlinear equation is in fact already solved in  the literature.  For a wide bibliography on the subject, we refer the reader for instance to \cite{EgSh}. 
The proof of Theorem \ref{theo:main} is concluded
\end{proof}

\section{Conjecture for the n-species model}

We  will state a conjecture about  hydrodynamic equations for the $n$-species model, briefly introduced in section  \ref{sec:clock}, in the so-called \emph{equidiffusive} case,
precisely described hereafter.
 
 \begin{defin}The $n$-species system  is said to be \emph{equidiffusive} whenever there exists a constant $\lambda$, such that, for all pairs $(k,l)$,
 \[
 \lim_{N\to\infty} \frac{\lambda^{kl}(N)}{N^2} = \lambda.
 \]
\end{defin}
Then, letting
\[
\alpha^{kl} \egaldef  \lim_{N\to\infty}N \log \left[\frac{\lambda^{kl}(N)}{\lambda^{lk}(N)}\right], 
\]
we assert the following hydrodynamic system holds.
\[
\frac{\partial\rho_k}{\partial t} = \lambda\left[\frac{\partial^2\rho_k}{\partial x^2}
+\frac{\partial}{\partial x}\Bigl(\sum_{l\ne k}\alpha^{lk}\rho_k\rho_l\Bigr)\right], \ k=1,\ldots,n.
\]
The idea is  to apply the functional approach  presented in this paper: this is the subject matter of an ongoing work.

\end{document}